\documentclass[a4paper,11pt]{amsart}
\usepackage[colorlinks, linkcolor=blue,anchorcolor=Periwinkle,
citecolor=blue,urlcolor=Emerald]{hyperref}
\usepackage[all]{xy}
\SelectTips{cm}{}
\usepackage{hyperref}
\usepackage{graphicx}
\usepackage{psfrag}
\usepackage{tikz}
\usepackage{tikz-cd}
\usepackage{mathtools}
\usepackage{amsmath,amssymb,tikz-cd}
\usepackage[export]{adjustbox}
\usepackage{calligra}

\usetikzlibrary{decorations.pathreplacing}
\usetikzlibrary{matrix,arrows}

\usepackage[french,english]{babel}

\numberwithin{equation}{subsection}
\newtheorem{theorem}{Theorem}
\newtheorem{definition}[theorem]{Definition}
\newtheorem{classification-theorem}[subsection]{Classification Theorem}
\newtheorem{decomposition-theorem}[subsection]{Decomposition Theorem}
\newtheorem{proposition-definition}[subsection]{Proposition-Definition}
\newtheorem{definition-proposition}[subsection]{Definition-Proposition}
\newtheorem{example-definition}[subsection]{Example-Definition}
\newtheorem{periodicity-conjecture}[subsection]{Periodicity Conjecture}
\newtheorem{lemma}[theorem]{Lemma}
\newtheorem{proposition}[theorem]{Proposition}

\newtheorem{example}[theorem]{Example}
\newtheorem{remark}[theorem]{Remark}

\newtheorem{Definition-Proposition}[subsection]{D\'efinition-Proposition}

\newcommand{\reminder}[1]{}

\newcommand{\rep}{\mathrm{rep}}

\newcommand{\tr}{\mathrm{tr}}

\newcommand{\dgcat}{\mathrm{dgcat}}
\newcommand{\Hqe}{\mathrm{Hqe}}

\renewcommand{\rep}{\mathrm{rep}}

\newcommand{\pretr}{\mathrm{pretr} }

\renewcommand{\ker}{\mathrm{ker} }

\newcommand{\Q}{\mathbb{Q}}

\newcommand{\iso}{\xrightarrow{_\sim}}

\newcommand{\pr}{\mathrm{pr}}

\newcommand{\Def}{\mathrm{def}\kern 0.1em}

\newcommand{\D}{\mathcal {D}}
\newcommand{\A}{\mathcal {A}}
\newcommand{\B}{\mathcal {B}}
\newcommand{\C}{\mathcal {C}}
\newcommand{\J}{\mathcal {J}}
\newcommand{\I}{\mathcal {I}}

\newcommand{\T}{\mathcal T}
\renewcommand{\P}{\mathcal P}

\newcommand{\Hom}{\mathrm{Hom}}

\newcommand{\Ext}{\mathrm{Ext}}

\newcommand{\Mor}{\mathrm{Mor}}

\newcommand{\ca}{{\mathcal A}}

\renewcommand{\phi}{\varphi}

\renewcommand{\Q}{\mathcal{Q}}

\begin{document}
\title[0-Auslander correspondence]{0-Auslander correspondence}
\author[Xiaofa Chen] {Xiaofa Chen}
\address{University of Science and Technology of China, Hefei, P.~R.~China}
\email{cxf2011@mail.ustc.edu.cn}

\date{\today}


\begin{abstract} We prove an analogue of Auslander correspondence for
exact dg categories whose $H^0$-category is $0$-Auslander in the
sense of Gorsky--Nakaoka--Palu.
\end{abstract}

\maketitle

Let $k$ be a commutative ring. 
We denote by $\dgcat$ the category of dg $k$-categories and by $\Hqe$ the localization of $\dgcat$ with respect to the quasi-equivalences, cf.~\cite{Tabuada05}. 
For a dg category $\A$, we denote by $H^0(\A)$ its 0-homology category, i.e.~the category with the same objects as $\A$ and whose morphism spaces are given by $H^0(\A)(A_1,A_2)=H^0(\A(A_1,A_2))$. We denote by $\D(\A)$ its derived category. The dg category $\A$ is {\em pretriangulated} if the essential image of
the Yoneda embedding $H^0(\A)\rightarrow \D(\A)$ is a triangulated subcategory.
For a dg category $\A$, we denote its {\em pretriangulated hull} by $\pretr(\A)$, cf.~\cite{BondalKapranov90, Drinfeld04, BondalLarsenLunts04}. Put $\tr(\A)=H^0(\pretr(\A))$.
We refer to \cite{Chen23} for the definition and the fundamental properties of an {\em exact dg category}.
Each pretriangulated category endowed with the class of all graded-split exact sequences is an exact 
dg category.  Also each dg category $\A$ such that $H^0(\A)$ is an additive category is an exact dg category when endowed with the class of all split exact sequences. 
A full dg subcategory $\B$ of an exact dg category $\A$ which is stable under extensions inherits from $\A$ a canonical exact structure. 
By \cite[Theorem 6.1]{Chen23}, for each connective exact dg category $\A$, there exists a universal exact morphism $F:\A\rightarrow \D^b_{dg}(\A)$ in $\Hqe$ from $\A$ to a pretriangulated dg category $\D^b_{dg}(\A)$ which induces a quasi-equivalence of exact dg categories from $\tau_{\leq 0}\A$ to $\tau_{\leq 0}\D'$ 
for an extension-closed dg subcategory $\D'$ of $\D^b_{dg}(\A)$. Put $\D^b(\A)=H^0(\D^b_{dg}(\A))$. For a triangulated category $\T$ and full subcategories $\C$ and $\D$ of $\T$, we denote by $\C\ast\D$ the full subcategory of $\T$ on the objects $X$ which admits a triangle $C\rightarrow X\rightarrow D\rightarrow \Sigma C$ in $\T$ where $C\in \C$ and $D\in\D$.

 It is shown in \cite{Chen23} that if $\A$ is an exact
dg $k$-category, then the category $H^0(\ca)$ inherits a natural extriangulated
structure in the sense of \cite{NakaokaPalu19}. An object $A\in \A$ is {\em projective}
 (resp.~{\em injective}) if it is projective (resp.~injective) 
 in the extriangulated category $H^0(\A)$. We define an exact dg
category $\A$ to be {\em $0$-Auslander} if $H^0(\A)$ is $0$-Auslander in 
the sense of \cite{GorskyNakaokaPalu23}. The following theorem is
a $0$-dimensional analogue of higher Auslander--Iyama correspondence
\cite{Auslander71, Iyama07}. It is inspired by recent work of
Fang--Gorsky--Palu--Plamondon--Pressland \cite{FangGorskyPaluPlamondonPressland23}.

\begin{theorem}[0-Auslander correspondence]\label{thm:0-Auslander correspondence}
There is a bijective correspondence between the following:
\begin{itemize}
\item[(1)] equivalence classes of connective exact dg categories $\A$ which are $0$-Auslander;
\item[(2)] equivalence classes of pairs $(\P,\I)$ consisting of 
\subitem- a connective additive dg category $\P$ and
\subitem- a dg subcategory $\I\subset \P$ such that $H^0(\I)$ is covariantly finite in $H^0(\P)$.
\end{itemize}
The bijection from $(1)$ to $(2)$ sends $\A$ to the pair $(\P,\I)$ formed by  the full dg subcategory $\P$ on the projectives of $\A$ and its full dg subcategory $\I$ of projective-injectives. The inverse bijection sends $(\P,\I)$ to the $\tau_{\leq 0}$-truncation of the  dg subcategory $(\P*\Sigma\P)\cap \ker\, \Ext^1(-,\I)$ of $\pretr(\P)$.
\end{theorem}
We will prove the theorem after Proposition~\ref{prop:0-Auslander}.
Clearly, under the correspondence of the theorem,
\begin{itemize}
\item[a)] the exact dg categories in $(1)$ whose corresponding extriangulated category is moreover reduced correspond to the pairs in $(2)$ where $\I=0$; 
\item[b)] the exact dg categories in $(1)$ with the trivial exact structure (so that each object is both projective and injective) correspond to the pairs in $(2)$ where $\I=\P$.
\end{itemize}
For a pair $(\P,\I)$ in $(2)$, we denote by $\A_{\P,\I}$ the dg subcategory $\P*\Sigma\P\cap \ker(\Ext^1(-,\I))$ of $\pretr(\P)$ and by $\Q_{\P,\I}$ (resp.~$\J_{\P,\I}$) be the full dg subcategory of $\A_{\P,\I}$ whose objects are the direct summands in $H^0(\A_{\P,\I})$ of objects in $H^0(\P)\subset H^0(\A_{\P,\I})$ (resp.~$H^0(\I)\subset H^0(\A_{\P,\I})$).
Clearly $\Q_{\P,\I}$ is stable under kernels of retractions in $\pretr(\Q_{\P,\I})\iso \pretr(\P)$.
Below we will show that $H^0(\A_{\P,\I})$ is {\em weakly idempotent complete}, i.e.~retractions of epimorphisms possess kernels, cf.~the proof of Theorem~\ref{thm:0-Auslander correspondence}.

For a connective exact dg category $\A$, we denote by $\overline{\A}$ the full dg subcategory of $\D^{b}_{dg}(\A)$ consisting of objects in the closure under kernels of retractions of $H^0(\A)$ in $\D^{b}(\A)$. Clearly $\overline{\A}$ is stable under extensions in $\D^b_{dg}(\A)$ and the thus inherited exact structure induces a quasi-equivalence $\D^b_{dg}(\overline{\A})\iso \D^b_{dg}(\A)$.

\begin{definition}\label{def:equ}
Two connective exact dg categories $\A$ and $\B$ are {\em equivalent} if there is a quasi-equivalence $\D^b_{dg}(\A)\iso \D^b_{dg}(\B)$ inducing a quasi-equivalence $\overline{\A}\iso \overline{\B}$.
 
Two pairs $(\P,\I)$ and $(\P',\I')$ in $(2)$ are {\em equivalent} if there is a quasi-equivalence $\pretr(\P)\iso \pretr(\P')$ inducing quasi-equivalences $\Q_{\P,\I}\iso \Q_{\P',\I'}$ and $ \J_{\P,\I}\iso \J_{\P',\I'}$.
\end{definition}
\begin{remark}\label{rmk: equivalence 0-Auslander}
A connective exact dg category $\A$ is $0$-Auslander if and only if $\overline{\A}$ is $0$-Auslander.
\end{remark}
\begin{proof}Let $\P$ (resp.~$\I$) be the full dg subcategory on the projectives (resp.~projective-injectives) of $\A$.
 Let $X\in \overline{\A}$. Then we have $Y$ and $Z$ in $\A$ such that $Y\iso X\oplus Z$. We have a conflation in $\A$
\[
M\rightarrow P\rightarrow Y
\]
where $P$ is projective in $H^0(\A)$. Then we have the following diagram in $\D^b(\A)$
\begin{equation}\label{dia:octa}
\begin{tikzcd}
M\ar[r]\ar[d,equal]&N\ar[r]\ar[d]&Z\ar[d]\ar[r]&\Sigma M\ar[d,equal]\\
M\ar[r]&P\ar[r]\ar[d]&Y\ar[d]\ar[r]&\Sigma M\\
&X\ar[r,equal]&X&
\end{tikzcd}
\end{equation}
Then it is clear that $N$ lies in $\A$ and hence $\overline{\A}$ has enough projectives.

To see that $\overline{\A}$ is hereditary, let us assume that $M$ is in $\P$. 
The canonical map $P\rightarrow Y\twoheadrightarrow Z$ extends to a triangle in $\D^b(\A)$
\begin{equation}\label{tri:Z}
P\rightarrow Z\rightarrow L\rightarrow \Sigma P.
\end{equation}
Since $\Sigma^{-1}L$ is an extension of $M$ and $X$ in $\D^b(\A)$, we have $\Hom_{\D^b(\A)}(\P,L)=0$. Since $L$ is an extension of $Z$ and $\Sigma P$, we have $\Hom_{\D^b(\A)}(L,\Sigma^2 N)=0$.

We have the following diagram in $\D^b(\A)$
\begin{equation}\label{dia:octa2}
\begin{tikzcd}
                             &P\ar[r,equal]\ar[d]                   &P\ar[d]                                                  &\\
P\ar[r] \ar[d,equal]&X\oplus Z\ar[r]\ar[d]                       & X\oplus  L \ar[d]\ar[r]                &\Sigma P\ar[d,equal] \\
P\ar[r]                   &\Sigma M\ar[r]                               &  \Sigma M\oplus \Sigma P\ar[r]                                                 &\Sigma P
\end{tikzcd}
\end{equation}
 and this gives rise to the following diagram in $\D^b(\A)$
 \begin{equation}\label{dia:octa3}
 \begin{tikzcd}
 P\ar[d,equal]\ar[r]&X\ar[d]\ar[r]&\Sigma N\ar[d]\ar[r]&\Sigma P\ar[d]\\
 P\ar[r]&X\oplus L\ar[d]\ar[r]&\Sigma M\oplus \Sigma P\ar[r]\ar[d]&\Sigma P\\
 &L\ar[r,equal]&L&
 \end{tikzcd}.
 \end{equation}
Since $\Sigma^{-1}L$ lies in $\overline{\A}$, by the third column we have that $N\in \overline{\A}$ and is projective in $\overline{\A}$.
 Hence we have that $\overline{\A}$ is hereditary.
 
 Suppose $X\in \overline{\A}$ is projective. 
 Since $\A$ has enough projectives, we have $Y$ and $Z$ in $\P$ such that $Y\iso X\oplus Z$.
 We have a conflation in $\A$
 \[
 Y\rightarrow I\rightarrow U
 \]
 where $I\in\I$ and $U$ is injective in $\A$.
 So we have the following diagram in $\D^b(\A)$
 \[
 \begin{tikzcd}
 X\ar[r,equal]\ar[d]&X\ar[d]&&\\
 Y\ar[r]\ar[d]&I\ar[r]\ar[d]&U\ar[d,equal]\ar[r]&\Sigma Y\ar[d]\\
 Z\ar[r]&V\ar[r]&U\ar[r]&\Sigma Z
 \end{tikzcd}
 \]
 Then $V\in \A$ and the second column shows that $X$ has dominant dimension at least 1.
 
\end{proof}
\begin{remark}\label{rmk:weakly idempotent complete}
If $H^0(\P)$ is weakly idempotent complete, then so is $H^0(\A_{\P,\I})$.
\end{remark}
\begin{proof}
Put $\A=\A_{\P,\I}$. Suppose we have $Y\iso X\oplus Z$ in $\tr(\P)$ where both $Y$ and $Z$ are in $\A$.
We have the diagram~\ref{dia:octa} in $\tr(\P)$ where both $M$ and $P$ are in $\P$. It is enough to show that $N\in\P$.
We have the triangle~\ref{tri:Z} and the diagram~\ref{dia:octa3} in $\tr(\P)$.
We have $M\oplus P\iso N\oplus \Sigma^{-1}L$. Since $N$ is an extension of $M$ and $Z$, it lies in $\A$ and in particular it lies in $\pr(\P)$.
Since $\P$ is weakly idempotent complete, we have that $N$ lies in $\P$.
\end{proof}
We start by showing that for a pair $(\P,\I)$ in $(2)$, the dg subcategory $\A_{\P,\I}$ of $\pretr(\P)$ inherits a canonical exact dg structure whose corresponding extriangulated category is 0-Auslander.
\begin{lemma}\label{lem:0-Auslander}
\begin{itemize}
\item[(1)] Let $\P$ be a connective additive dg category and $\I\subseteq \P$ a full dg subcategory such that $H^0(\I)$ is covariantly finite in $H^0(\P)$. 
Denote by $\pr(\P)$ the full dg subcategory of $\pretr(\P)$ consisting of the objects $X$ in $\P*\Sigma \P$, i.e.~there exist a triangle 
\begin{equation}\label{tri:X}
P_1\xrightarrow{f} P_2\xrightarrow{g} X\xrightarrow{h} \Sigma P_1
\end{equation}
in $\tr(\P)$ where $P_i$, $i=1$, $2$, are objects in $\P$ and we omit the symbol $^{\wedge}$.
Then the full subcategory $H^0(\A_{\P,\I})$ of $\tr(\P)$ is extension-closed and the inherited extriangulated structure is 0-Auslander. 
Put $\A=\tau_{\leq 0}\A_{\P,\I}$.
\item [(2)] 
We have that $\Q_{\P,\I}$ is the full dg subcategory of projectives in $\A_{\P,\I}$ and $\J_{\P,\I}$ is the full dg subcategory of projective-injectives in $\A_{\P,\I}$. Moreover $(\Q_{\P,\I},\J_{\P,\I})$ is equivalent to $(\P,\I)$.
\item [(3)]The dg derived category of $\A$ is $\pretr(\P)$.
\end{itemize}
\end{lemma}
\begin{proof}
Since $\P$ is connective, $\pr(\P)$ is extension-closed in $\pretr(\P)$ and hence $\A$ is also extension-closed in $\pretr(\P)$.

It is clear that $\Q_{\P,\I}$ is the full dg subcategory of projective objects: if $X$ is projective in $\A$, then the triangle \ref{tri:X} splits.

It is also clear that the objects in $\I$ are projective-injective.
 
Let us show that $H^0(\A)$ has dominant dimension at least $1$: if $X$ is projective in $\A$, then the triangle \ref{tri:X} splits.
Since $H^0(\I)$ is covariantly finite in $H^0(\P)$, we have a triangle
\begin{equation}\label{tri:I}
P_2\xrightarrow{u} I\xrightarrow{v} Y\xrightarrow{w} \Sigma X
\end{equation}
where $I\in\I$ and $u$ is a left $\I$-approximation of $P_2$. 
Then it is clear that $Y$ lies in $\A$ and that $Y$ is injective in $H^0(\A)$: this follows from the fact that $\Hom_{\tr(\P)}(A,\Sigma^2A')=0$ for $A$ and $A'$ in $\A$ and that $I$ is injective in $H^0(\A)$.
We have the following diagram
\begin{equation}
\begin{tikzcd}\label{dia:X}
X\ar[r,equal]\ar[d]&X\ar[d]&&\\
P_2\ar[r,"u"]\ar[d]&I\ar[r,"v"]\ar[d]&Y\ar[d,equal]\ar[r]&\Sigma P_2\ar[d]\\
P_1\ar[r]&U\ar[r]&Y\ar[r]&\Sigma P_1
\end{tikzcd}.
\end{equation}
Since $U$ is an extension of $P_1$ and $Y$, it lies in $\A$. It is clear that $U$ is injective in $\A$.
So the middle column shows that the dominant dimension of $X$ is a least 1.

If $X$ is moreover injective, then the triangle in the middle column splits and hence $X$ lies in $\J_{\Q,\I}$.
\end{proof}

 A dg $\A$-module $M$ is {\em quasi-representable} if it is quasi-isomorphic to $A^\wedge=\Hom_{\A}(?,A)$ for some $A \in \A$. 
 For dg categories $\A$ and $\B$, 
 we denote by $\rep(\B,\A)$ the full subcategory of $\D(\A\otimes^{\mathbb L} \B^{op})$ 
 whose objects are the dg bimodules $X$ such that $X(-,B)$ is quasi-representable for each object $B$ of $\B$. 
 We define the canonical dg enhancement of $\rep(\B,\A)$ by $\rep_{dg}(\B,\A)$. 
 By \cite[Theorem 6.43]{Chen23} $\rep_{dg}(\B,\A)$ is canonically an exact dg category 
 when $\A$ is an exact dg category and $\B$ is a connective dg category.
 
 Let $I_2$ be the quiver with two vertices $0$ and $1$ and a unique arrow $1\rightarrow 2$.
 We view the path $k$-category $k\I_2$ as a dg $k$-category concentrated in degree 0.
The dg category $\pr(\P)$ is quasi-equivalent to the dg category $\Mor(\P)\iso \rep_{dg}(kI_2, \P)$ whose exact structure is induced from the trivial exact structure on $\P$.
\begin{proposition}\label{prop:0-Auslander}
Let $\A'$ be a connective exact dg category such that $H^0(\A')$ is $0$-Auslander. 
Suppose that $\P\subset \A'$ is the full dg subcategory of projective objects and $\I\subset \P$ is the full dg subcategory of projective-injective objects (so $H^0(\I)$ is covariantly finite in $H^0(\P)$).

Then $\A'$ is equivalent (in the sense of Definition~\ref{def:equ}) to $\A=\tau_{\leq 0}(\A_{\P,\I})$.
When $\I=0$, we have that $\A'$ is actually quasi-equivalent to $\A$.
\end{proposition}
\begin{proof}
We have the following diagram
\[
\begin{tikzcd}
\P\ar[r]\ar[d,tail]&\A'\ar[r,tail]&\D^b_{dg}(\A')\\
\pretr(\P)\ar[rru,"\mu"swap]&&
\end{tikzcd}
\]
The induced morphism $\P\rightarrow \D^b_{dg}(\A')$ is quasi-fully faithful and $\D^b(\A')$ is generated as a triangulated category by the objects in $\P$. 
Therefore the morphism $\mu:\pretr(\P)\rightarrow \D^b_{dg}(\A')$ is a quasi-equivalence.
Recall that $\pretr(\P)$ is the dg derived category of $\A$.

It is enough to show that the quasi-essential image of $\A\subset \pretr(\P)$ under the morphism $\mu$ is in the closure of $\A'$ under kernels of retractions in $\D^b_{dg}(\A')$.

Let $X\in \D^b(\A')$ be an object with a triangle \ref{tri:X}. Assume that $\Hom_{\D^b(\A')}(X,\Sigma \I)=0$. 
We claim that $X$ lies in the closure of $\A'$ under kernels of retractions in $\D^b(\A')$.
Indeed, the object $P_1$ in the triangle \ref{tri:X} admits a triangle as follows
\[
P_1\rightarrow I\rightarrow K\rightarrow \Sigma P_1
\]
where $I\in \I$ and $K$ is injective in $H^0(\A')$.
We have the following diagram
\begin{equation}
\begin{tikzcd}\label{dia:Y}
P_1\ar[r,"f"]\ar[d]&P_2\ar[r,"g"]\ar[d]&X\ar[d,equal]\ar[r]&\Sigma P_1\ar[d]\\
I\ar[r]\ar[d]&U\ar[r]\ar[d]&X\ar[r]&\Sigma I\\
K\ar[r,equal]&K& &
\end{tikzcd}.
\end{equation}
By assumption we have that $\Hom(X,\Sigma I)=0$ and therefore $X\oplus I$ is isomorphic to $U$, an extension of $P_2$ and $K$ (so $U\in \A'$). 
So $X$ lies in $\overline{\A'}$.

On the other hand, it is clear that each object in $\A'$ is from an object in $\A$.
\end{proof}
\begin{proof}[Proof of Theorem~\ref{thm:0-Auslander correspondence}]\label{proof:main theorem}
By Remark~\ref{rmk: equivalence 0-Auslander}, the property of being 0-Auslander is invariant under the equivalence relation introduced in Definition~\ref{def:equ}.
Notice that for a 0-Auslander exact dg category, the inclusion $\A\rightarrow \overline{\A}$ induces a quasi-equivalence $\D^b_{dg}(\A)\iso \D^b_{dg}(\overline{\A})$ and $\D^b_{dg}(\A)$ is quasi-equivalent to the pretriangulated hull of the full dg subcategory of $\A$ on the projectives. 
Therefore the map sending the equivalence class of $\A$ to the equivalence class $(\P,\I)$ is well-defined.

For a pair $(\P,\I)$ in (2), we claim that $H^0(\A_{\P,\I})$ is weakly idempotent complete. We have a quasi-equivalence $\pretr(\P)\iso \pretr(\Q_{\P,\I})$. Since $H^0(\Q_{\P,\I})$ is stable under kernels of retractions in $\tr(\Q_{\P,\I})$, we see that $H^0(\A_{\Q_{\P,\I},\J_{\P,\I}})$ is weakly idempotent complete by Remark~\ref{rmk:weakly idempotent complete}.  It is not hard to check that $\A_{\P,\I}=\A_{\Q_{\P,\I},\J_{\P,\I}}$. Therefore $H^0(\A_{\P,\I})$ is weakly idempotent complete and the map sending the equivalence class of a pair $(\P,\I)$ to the equivalence class of $\tau_{\leq 0}\A_{\P,\I}$ is well-defined.

Combining Lemma~\ref{lem:0-Auslander} and Proposition~\ref{prop:0-Auslander}, we see that the above maps are inverse to each other.
\end{proof}
Recall that an extriangulated category is {\em algebraic}, if it is equivalent as an extriangulated category to $H^0(\A)$ for an exact dg category $\A$, cf.~\cite[Proposition-Definition 6.20]{Chen23}. Since $H^0(\tau_{\leq 0}\A)$ and $H^0(\A)$ are equal as extriangulated triangulated categories, we may assume in the above definition that $\A$ is connective. Inspired by recent results of Fang--Gorsky--Palu--Plamondon--Pressland~\cite{FangGorskyPaluPlamondonPressland23a}, we have the following example which relates an algebraic 0-Auslander extriangulated category to a homotopy category of two-term complexes.
\begin{example}\label{exm:idealquotientalgebraic}
\rm{
Let $(\P,\I)$ be a pair in Theorem~\ref{thm:0-Auslander correspondence} (2) where $\P$ is weakly idempotent complete. Put $\A=\tau_{\leq 0}\A_{\P,\I}$. Let $\J$ be the full dg subcategory of injectives in $\A$.
 Then we have a canonical equivalence of extriangulated categories $H^0(\A)/(\J\rightarrow \P)\iso \mathcal K^{[-1,0]}(H^0(\P)/[\I])$.
}
\end{example}
\begin{proof}
We keep the same notations $\P$ and $\I$ for the corresponding subcategories $H^0(\P)$ and $H^0(\I)$ of $H^0(\A)$.
 Let $\A/\I$ be the dg quotient. By \cite[Theorem 6.35]{Chen23} $\A/\I$ carries a canonical exact structure whose corresponding extriangulated category is $H^0(\A)/[\I]$ which is again {0-Auslander}.
So we may replace $\A$ by $\A/\I$ and assume $\I=0$. Then $\J=\Sigma \P$ and $\Hom_{H^0(\A)}(\P,\J)=0$.

The canonical dg functor $\P\rightarrow H^0(\P)$ induces a canonical exact dg functor 
\[
\A_{\P,0}\rightarrow \A_{H^0(\P),0}
\]
which is quasi-dense.
By taking 0-homology on both sides, it induces a canonical functor
\[
H^0(\A)\rightarrow \mathcal K^{[-1,0]}(H^0(\P)).
\]
In $\mathcal K^{[-1,0]}(H^0(\P))$ we have $\Hom(\Sigma \P,\P)=0$. Therefore it induces a canonical functor 
\[
F: \C=H^0(\A)/(\J\rightarrow \P)\iso \mathcal K^{[-1,0]}(H^0(\P)/[\I])=\C'.
\]
Let $X$ and $Y$ be objects in $\A$. Assume we have triangles in $\tr(\P)$
\begin{equation}\label{tri:X'}
P_1\xrightarrow{f} P_0\rightarrow X\rightarrow \Sigma P_1
\end{equation}
and 
\begin{equation}\label{tri:Y'}
Q_1\rightarrow Q_0\rightarrow Y\rightarrow \Sigma Q_1.
\end{equation}
We have 
\[
\Hom_{\C}(\P,\P)\iso \Hom_{H^0(\A)}(\P,\P)\iso\Hom_{\C'}(\P,\P)
\]
and 
\[
\Hom_{\C}(\P,\Sigma\P)\iso \Hom_{H^0(\A)}(\P,\Sigma \P)\iso \Hom_{\C'}(\P,\Sigma\P)=0.
\]
We apply the cohomological functor $\Hom_{\tr(\P)}(-,\P)$ to the triangle \ref{tri:X'} and we get a long exact sequence of Hom spaces in $\tr(\P)$
\[
\ldots\rightarrow \Hom(\Sigma P_1,\P)\rightarrow \Hom(X,\P)\rightarrow \Hom(P_0,\P)\xrightarrow{\Hom(f,\P)} \Hom(P_1,\P)\rightarrow \Hom(X,\Sigma \P)\rightarrow 0. 
\]
By direct inspection we see that $\Hom_{\C}(X,\P)$ is isomorphic to the kernel of the map $\Hom(f,\P)$ and that $\Hom_{H^0(\A)}(X,\Sigma \P)=\Hom_{\C}(X,\Sigma \P)$. 
Therefore we have 
\[
\Hom_{\C}(X,\P)\iso \Hom_{\C'}(FX,F \P)
\]
 and $\Hom_{\C}(X,\Sigma \P)\iso \Hom_{\C'}(FX,F \Sigma \P)$.
Now from the $\mathbb E$-triangle in $\C$ given by \ref{tri:Y'}, we have a long exact sequence
\[
\Hom_{\C}(X,Q_1)\rightarrow \Hom_{\C}(X,Q_0)\rightarrow \Hom_{\C}(X,Y)\rightarrow \mathbb E(X,Q_1)\rightarrow \mathbb E(X,Q_0)\rightarrow \mathbb E(X,Y)\rightarrow 0.
\]
We have $\mathbb E(X,Q_1)\iso \Hom_{H^0(\A)}(X,\Sigma Q_1)=\Hom_{\C}(X,\Sigma Q_1)$. By the Five-Lemma, we see that the functor $F$ is fully faithful and is an equivalence of extriangulated categories.
\end{proof}
\begin{remark}
Example~\ref{exm:idealquotientalgebraic} is the main motivation for the study of $0$-Auslander correspondence: Pressland \cite{FangGorskyPaluPlamondonPressland23} showed Example~\ref{exm:idealquotientalgebraic} in certain special cases, and the natural question to extend it to the case of any algebraic 0-Auslander extriangulated category leads to these notes. 
Fang--Gorsky--Palu--Plamondon--Pressland also showed Example~\ref{exm:idealquotientalgebraic} \cite[Theorem 4.2]{FangGorskyPaluPlamondonPressland23a} and the implication from (2) to (1) of Theorem~\ref{thm:0-Auslander correspondence} \cite[Theorem 4.1]{FangGorskyPaluPlamondonPressland23a}.
Combining Theorem~\ref{thm:0-Auslander correspondence} and Example~\ref{exm:idealquotientalgebraic}, we resolve a problem raised in \cite[Problem 3.6]{FangGorskyPaluPlamondonPressland23a} in the algebraic case.
\end{remark}
	\def\cprime{$'$} \def\cprime{$'$}
	\providecommand{\bysame}{\leavevmode\hbox to3em{\hrulefill}\thinspace}
	\providecommand{\MR}{\relax\ifhmode\unskip\space\fi MR }
	\providecommand{\MRhref}[2]{%
		\href{http://www.ams.org/mathscinet-getitem?mr=#1}{#2}
	}
	\providecommand{\href}[2]{#2}
%


	\bibliographystyle{amsplain}
	\bibliography{stanKeller}

\end{document}